\documentclass[reqno,12pt]{amsart}
\usepackage{amsmath,amsthm,amsfonts,amssymb}
\usepackage[utf8x]{inputenc}
\usepackage[english]{babel}
\usepackage[normalem]{ulem} 
\usepackage{graphicx}

\date{} 
\newtheorem{theorem}{Theorem}[section]

\newtheorem{proposition}[theorem]{Proposition}

\newtheorem{corollary}[theorem]{Corollary}
\newtheorem{example}[theorem]{Example}
\numberwithin{equation}{section}

\begin{document}

\title[]{On the problem of optimal fair exchange}
\maketitle

\begin{center}
A.V. Kolesnikov$^1$, S.N. Popova$^1$ 
\let\thefootnote\relax\footnotetext{$^1$ National Research University Higher School of Economics. 
This work is an output of a research project implemented as part of the Basic Research Program at the National Research University Higher School of Economics (HSE University)}
\end{center}

\vskip .2in

{\bf Abstract.}
We consider the problem of optimal exchange which can be formulated as a kind of optimal transportation problem. The existence of an optimal solution and a duality theorem for the optimal exchange problem are proved in case of completely regular topological spaces. We show the connection between the problem of optimal exchange and the optimal transportation problem with density constraints. With the use of this connection we obtain a formula for the optimal value in the problem of optimal exchange.


\vskip .2in

\section{Introduction}

Recently various new modifications of the Monge-Kantorovich optimal transportation problem have been widely studied  
 (see \cite{ABS21}, \cite{B22}, \cite{BKS}, \cite{BPR}). 
Recall that for Borel probability measures $\mu$ and $\nu$ on topological spaces $X$ and $Y$ respectively 
and a nonnegative Borel measurable function $h$ on~$X\times Y$ the classical Kantorovich optimal transportation problem 
concerns the minimization of the integral  
$$
\int_{X \times Y} h(x, y) \sigma (dx \, dy)
$$
over all measures $\sigma$ from the set $\Pi(\mu,\nu)$ consisting of all Borel probability measures on $X\times Y$
with projections $\mu$ and $\nu$ on the factors, that is, $\sigma (A\times Y)=\mu(A)$ and $\sigma (X\times B)=\nu(B)$
for all Borel sets $A\subset X$ and $B\subset Y$. 
Denote by $Pr_X(\sigma)$ and $Pr_Y(\sigma)$ the projections of a measure $\sigma$ on $X$ and $Y$ respectively. 
The measures $\mu$ and $\nu$ are called marginal distributions or marginals and $h$ is called a cost function. 
Basic information about the Monge-Kantorovich problem can be found in \cite{ABS21}, \cite{BKS}, \cite{RR}, \cite{Sant}, \cite{V}. In \cite{G} and \cite{Sant} economic applications of the optimal transportation problem are addressed. 

In this work we consider the problem of optimal exchange which can be formulated as a kind of optimal transportation problem. 
Let $X$ be the set of trading participants and let $S$ be the set of goods. Let further $Y = X$ be a copy of $X$. 
We assume that a unit of the good $s \in S$ has a cost $c(s)$, where 
$c \colon S \to \mathbb R$ is a positive cost function. 

Suppose we are given a discrete distribution $\pi^{+}$ on $X \times S$ and a discrete distribution $\pi^{-}$ 
on $Y \times S$. 
For any $i \in X$, $s \in S$ the value $\pi^{+}_{i, s}$ is equal to the amount of the good $s$ 
that the $i$-th participant is ready to share with other participants. Similarly $\pi^{-}_{i, s}$ is equal 
to the amount of the good $s$ that the $i$-th participant wishes to receive. 
Denote by $\gamma_{i, j, s}$ the amount of the good $s$ that the $i$-th participant sends to the $j$-th participant.

We obtain the following problem: 
\begin{equation} \label{prob_discr}
\sum_{i, j, s} c(s) \gamma_{i, j, s} \to \max, 
\end{equation}
$$
\sum_{j} \gamma_{i, j, s} \le \pi^{+}_{i, s}, \quad \sum_{i} \gamma_{i, j, s} \le \pi^{-}_{j, s}, 
$$
$$
\sum_{j, s} c(s) \gamma_{k, j, s} = \sum_{i, s} c(s) \gamma_{i, k, s}. 
$$
The last equality is the balance condition: for any participant $k$ the total cost of the sent items of goods 
is equal to the total cost of the received items of goods. 

Let us consider the continuous setting of the optimal exchange problem which is a natural generalization of the discrete 
problem (\ref{prob_discr}). 
Let $X = Y$ and $S$ be topological spaces. Let $\pi^{+}$ be a nonnegative Borel measure on $X \times S$ 
and let $\pi^{-}$ be a nonnegative Borel measure on $Y \times S$. 
Assume also that $c \colon S \to \mathbb R$ is a positive Borel measurable function. 
The optimal exchange problem concerns the maximization of the integral functional 
\begin{equation} \label{prob_cont}
\int_{X \times Y \times S} c(s) d\gamma \to \max 
\end{equation}
over all nonnegative Borel measures $\gamma$ on $X \times Y \times S$ satisfying the constraints 
\begin{equation} \label{ineq}
Pr_{X \times S}(\gamma) \le \pi^{+}, \quad Pr_{Y \times S}(\gamma) \le \pi^{-}, 
\end{equation}
\begin{equation} \label{eq}
\int_{Y \times S} c(s) d\gamma = \int_{X \times S} c(s) d\gamma. 
\end{equation}

Denote by
\begin{multline*}
\Gamma(\pi^{+}, \pi^{-}, c) = \bigl\{\gamma \in \mathcal M_{+}(X \times Y \times S): 
Pr_{X \times S}(\gamma) \le \pi^{+}, \quad Pr_{Y \times S}(\gamma) \le \pi^{-}, \\
\int_{Y \times S} c(s) d\gamma = \int_{X \times S} c(s) d\gamma \bigr\} 
\end{multline*}
the set of admissible measures $\gamma \in \mathcal M_{+}(X \times Y \times S)$ in the problem of optimal exchange. 

Let us introduce the notation and terminology that will be used in this paper. 
The Borel $\sigma$-algebra of the topological space $X$ is denoted by $\mathcal{B}(X)$.
A nonnegative Borel measure $\mu$ on a topological space $X$ is called Radon if for every Borel set $B$ and every 
$\varepsilon>0$ there exists a compact set $K\subset B$ with $\mu(B\backslash K)<\varepsilon$. 
The set of Radon nonnegative measures on $X$ is denoted by $\mathcal M_{+}(X)$.
A family $M\subset \mathcal M_{+}(X)$ is called uniformly tight if for every $\varepsilon>0$ there exists a compact set $K$ 
such that $\mu(X\backslash K)<\varepsilon$ for all $\mu\in M$.

The space $\mathcal M(X)$ of bounded signed Radon measures on $X$ is endowed with the weak topology generated by the seminorms of the form
$$
\mu \mapsto \biggl|\int_X f\, d\mu\biggr|,
$$
where $f$ is a bounded continuous function on $X$.

For a signed measure $\mu \in \mathcal M(X)$ denote by $\mu_{+}$ and $\mu_{-}$ the positive and negative parts of the measure $\mu$ in the Hahn decomposition $\mu = \mu_{+} - \mu_{-}$.  
For nonnegative measures $\mu, \nu \in \mathcal M_{+}(X)$ we denote $\mu \land \nu = \mu - (\mu - \nu)_{+}$. 

For any measure $\sigma \in \mathcal M_{+}(X \times Y)$ with the projection $\mu$ on $X$ 
conditional measures $\sigma^x$, $x \in X$, are nonnegative Radon measures on $Y$ with the property
$$
\int_{X\times Y}\varphi(x)\psi(y)\, \sigma(dx\, dy)
=\int_X \int_Y \varphi(x)\psi(y)\, \sigma^x(dy)\, \mu(dx)
$$
for all bounded Borel functions $\varphi$ on $X$ and $\psi$ on $Y$, where it is also assumed that the function
$x\mapsto \sigma^x(B)$ is $\mu$-measurable for every Borel set $B\subset Y$.
Conditional measures exist under rather broad assumptions, for example, in case of Souslin spaces. 
The fact that a measure $\sigma$ with the projection $\mu$ on $X$ has conditional measures 
is usually expressed as the equality 
$$
\sigma(dx\, dy)=\sigma^x(dy)\, \mu(dx). 
$$

We shall need below the so-called gluing lemma (see \cite{BK}). Let $X_1, X_2, X_3$ be completely regular topological spaces  and let $\mu_{1, 2}$ and $\mu_{2, 3}$ be nonnegative Radon measures on $X_1 \times X_2$ and $X_2 \times X_3$ respectively, such that their projections on $X_2$ coincide. Then there exists a nonnegative Radon measure $\mu$ on $X_1 \times X_2 \times X_3$ such that its projection on $X_1 \times X_2$ is $\mu_{1, 2}$ and its projection on $X_2 \times X_3$ is $\mu_{2, 3}$. 

If the measure $\mu_{2, 3}$ has conditional measures 
$$
\mu_{2, 3}(dx_2 \, dx_3) = \mu_{2, 3}^{x_2}(dx_3) \, \mu_2(dx_2),
$$ 
where $\mu_2$ is the projection of $\mu_{2, 3}$ on $X_2$, 
then we can take the measure
$$
\mu(dx_1 \, dx_2 \, dx_3) = \mu_{2, 3}^{x_2}(dx_3) \, \mu_{1, 2}(dx_1 \, dx_2)
$$
as a gluing of measures $\mu_{1, 2}$ and $\mu_{2, 3}$. 
 
In this paper we prove the existence of an optimal solution to the problem of optimal exchange (\ref{prob_cont})-(\ref{eq}) (see Section 2). Furthermore, in Section 2 the dual problem for the problem (\ref{prob_cont})-(\ref{eq}) is formulated 
and the duality formula is proved in case of completely regular topological spaces. 
In Section 3 we show the connection between the problem of optimal exchange and the optimal transportation problem with density constraints. With the use of this connection we propose a method of constructing the optimal solution in the optimal exchange problem. In Section 4 under the assumption that the supports of the measures $\pi^{+}$ and $\pi^{-}$ are disjoint
the optimal exchange problem is reduced to a Kantorovich optimal transportation problem with density constraints. 

\section{Duality in the optimal exchange problem} 

First we prove the existence of an optimal solution to the optimal exchange problem for Radon measures on completely regular topological spaces.  

\begin{theorem} \label{th_exist}  
Let $X = Y$ and $S$ be completely regular topological spaces, 
let $\pi^{+} \in \mathcal M_{+}(X \times S)$, $\pi^{-}\in \mathcal M_{+}(Y \times S)$ be Radon nonnegative measures and let $c \colon S \to \mathbb R$ be a positive Borel measurable cost function. 
Then the maximum in the problem (\ref{prob_cont})-(\ref{eq}) is attained, that is, there exists an optimal measure $\gamma \in \Gamma(\pi^{+}, \pi^{-}, c)$. 
\end{theorem}

\begin{proof}
Note that the problem (\ref{prob_cont})-(\ref{eq}) with an arbitrary cost function $c(s)$ can be reduced to the problem with constant cost function $c(s) = 1$. 
Indeed, set $\tilde \gamma = c(s) \gamma$. Then the problem (\ref{prob_cont})-(\ref{eq}) is equivalent to the following problem of optimal exchange with unit cost function:
$$
\int_{X \times Y \times S} d \tilde \gamma \to \max, 
$$
$$
Pr_{X \times S}(\tilde \gamma) \le c(s) \pi^{+}, \quad Pr_{Y \times S}(\tilde \gamma) \le c(s) \pi^{-}, 
$$
$$
\int_{Y \times S} d \tilde \gamma = \int_{X \times S} d \tilde \gamma. 
$$

Further without limitation of generality we assume that $c(s) = 1$. 

Let us prove that the set $\Gamma(\pi^{+}, \pi^{-}, c)$ is compact in the weak topology. For this purpose we prove that 
the set $\Gamma(\pi^{+}, \pi^{-}, c)$ is uniformly tight. 
Since the measures $\pi^{+}$ and $\pi^{-}$ are Radon, for any $\varepsilon > 0$ there exist compact sets 
$K_1 \subset X \times S$ and $K_2 \subset Y \times S$ such that $\pi^{+}((X \times S) \setminus K_1) < \varepsilon/2$ and
$\pi^{-}((Y \times S) \setminus K_2) < \varepsilon/2$. Let
$$K = \{(x, y, s) \in X \times Y \times S: (x, s) \in K_1, (y, s) \in K_2\}.$$
Then $K$ is a compact subset of $X \times  Y \times S$ and for any measure $\gamma \in \Gamma(\pi^{+}, \pi^{-}, c)$ we have
$$
\gamma((X \times Y \times S) \setminus K) \le \pi^{+}((X \times S) \setminus K_1) + \pi^{-}((Y \times S) \setminus K_2) < \varepsilon. 
$$
Thus, the set $\Gamma(\pi^{+}, \pi^{-}, c)$ is uniformly tight. Furthermore, $\Gamma(\pi^{+}, \pi^{-}, c)$ 
is uniformly bounded in variation and closed in the weak topology. Therefore, by Prokhorov's theorem (see \cite{B07}) 
the set $\Gamma(\pi^{+}, \pi^{-}, c)$ is compact in the weak topology. 
The functional $\gamma \mapsto \int c(s) d\gamma$ is continuous in the weak topology and, therefore, 
attains a maximum on the compact set $\Gamma(\pi^{+}, \pi^{-}, c)$. 

\end{proof}

Let us formulate the dual problem for the problem of optimal exchange (\ref{prob_cont})-(\ref{eq}):
we aim to minimize the functional 
\begin{equation} \label{prob_dual}
\int_{X \times S} f(x, s) d\pi^{+} + \int_{Y \times S} g(y, s) d \pi^{-} \to \min
\end{equation}
under the constraints
\begin{equation} \label{cond_dual1}
f(x, s) + g(y, s) + c(s) h(x) - c(s) h(y) \ge c(s),
\end{equation}
where
\begin{equation} \label{cond_dual2}
f \colon X \times S \to \mathbb R_{+}, \quad g \colon Y \times S \to \mathbb R_{+}, \quad h \colon X \to \mathbb R 
\end{equation}
are Borel measurable functions. 

\begin{proposition}
The maximum in the primal problem (\ref{prob_cont})-(\ref{eq}) is not greater than the infimum in the dual problem  (\ref{prob_dual})-(\ref{cond_dual2}). 
\end{proposition}

\begin{proof}
Let $\gamma \in \Gamma(\pi^{+}, \pi^{-}, c)$ and Borel measurable functions $f \colon X \times S \to \mathbb R_{+}$, 
$g \colon Y \times S \to \mathbb R_{+}$, $h \colon X \to \mathbb R$ be such that
$f(x, s) + g(y, s) + c(s) h(x) - c(s) h(y) \ge c(s)$ for all $x \in X$, $y \in Y$, $s \in S$. Then
\begin{multline*}
\int_{X \times Y \times S} c(s) d\gamma \le \int_{X \times Y \times S} (f(x, s) + g(y, s) + c(s) h(x) - c(s) h(y)) d\gamma 
= \\ = \int_{X \times S} f(x, s) d Pr_{X \times S}(\gamma) + \int_{Y \times S} g(y, s) d Pr_{Y \times S}(\gamma) + \\ +
\int_X h(x) \int_{Y \times S} c(s) d\gamma - \int_Y h(y) \int_{X \times S} c(s) d\gamma \le \\ \le
\int_{X \times S} f(x, s) d\pi^{+} + \int_{Y \times S} g(y, s) d \pi^{-}
\end{multline*}
in virtue of (\ref{ineq}) and (\ref{eq}). 
\end{proof}

We prove the equality of the optimal values in the primal and dual problem for Radon measures on completely regular topological spaces. 

\begin{theorem} \label{th_dual}
Let $X = Y$ and $S$ be completely regular topological spaces, 
let $\pi^{+} \in \mathcal M_{+}(X \times S)$, $\pi^{-}\in \mathcal M_{+}(Y \times S)$ be Radon nonnegative measures
and let $c \colon S \to \mathbb R$ be a positive Borel measurable cost function. 
Then the maximum in the problem (\ref{prob_cont})-(\ref{eq}) is equal to the infimum in the dual problem 
(\ref{prob_dual})-(\ref{cond_dual2}). 
\end{theorem}

\begin{proof}
As in the proof of Theorem \ref{th_exist}, without limitation of generality we may assume that $c(s) = 1$ for all $s \in S$. 

First we prove Theorem \ref{th_dual} in the case where $X = Y$ and $S$ are compact topological spaces. 
We shall use the Fenchel-Rockafellar duality theorem. Let $U = C_b(X \times Y \times S)$. 
By Riesz theorem we have $U^* = \mathcal M(X \times Y \times S)$. 
Let us define the functional $\Theta$ on $C_b(X \times Y \times S)$ in the following way. 
For any function $u \in C_b(X \times Y \times S)$ we set $\Theta(u) = 0$, if there exists a function $h \in C_b(X)$ such that 
\begin{equation} \label{def_theta}
u(x, y, s) + h(x) - h(y) \ge 1 \quad \forall x \in X, y \in Y, s \in S, 
\end{equation}
and we set $\Theta(u) = +\infty$ otherwise. 
Next we define
\begin{multline*}
\Xi(u) = \inf \Bigl\{\int_{X \times S} f(x, s) d\pi^{+} + \int_{Y \times S} g(y, s) d \pi^{-}: f \in C_b(X \times S), g \in C_b(Y \times S), \\  f \ge 0, g \ge 0, \, u(x, y, s) = f(x, s) + g(y, s) \Bigr\}.
\end{multline*}
(we set $\Xi(u) = +\infty$ if such functions $f$ and $g$ do not exist). 
The functionals $\Theta$ and $\Xi$ are convex on $C_b(X \times Y \times S)$. For the function $u_0$ which is equal to a constant greater than 1, we have $\Theta(u_0) = 0$, $\Xi(u_0) < +\infty$ 
and $\Theta$ is continuous at $u_0$ (since $\Theta(u) = 0$ for all functions $u \ge 1$).  
By Fenchel-Rockafellar theorem (see \cite{V})
$$
\inf_{u \in C_b(X \times Y \times S)} (\Theta(u) + \Xi(u)) = \max_{\gamma \in \mathcal M(X \times Y \times S)} (-\Theta^*(-\gamma) - \Xi^*(\gamma)). 
$$

Note that $\inf_{u \in U} (\Theta(u) + \Xi(u))$ is equal to the infimum in the dual problem (\ref{prob_dual})-(\ref{cond_dual2}). Let us show that the maximum on the right-hand side is equal to the maximum in the primal problem (\ref{prob_cont})-(\ref{eq}). 
For this we prove that if the value $-\Theta^*(-\gamma) - \Xi^*(\gamma)$ is finite for a measure $\gamma \in \mathcal M(X \times Y \times S)$, then $\gamma \in \Gamma(\pi^{+}, \pi^{-}, c)$ 
and
$$
-\Theta^*(-\gamma) - \Xi^*(\gamma) = \int_{X \times Y \times S} d\gamma. 
$$

First we consider the value $-\Theta^*(-\gamma)$. Let us prove that  
$-\Theta^*(-\gamma) = \int_{X \times Y \times S} d \gamma$, if the measure $\gamma$ is nonnegative and $\int_{Y \times S} d\gamma = \int_{X \times S} d\gamma$,  
and $-\Theta^*(-\gamma) = -\infty$ otherwise.   

We have
$$
- \Theta^*(-\gamma) = - \sup_{u \in U} \Bigl(-\int_{X \times Y \times S} u d\gamma - \Theta(u)\Bigr) = 
\inf_{u \in U} \Bigl(\int_{X \times Y \times S} u d\gamma + \Theta(u) \Bigr).
$$

If the measure $\gamma$ is not nonnegative, then there exists a function $\varphi \in C_b(X \times Y \times S)$ such that
$\varphi \ge 0$ and $\int_{X \times Y \times S} \varphi d\gamma < 0$. Let $u(x, y, s) = 1 + t \, \varphi(x, y, s)$, where $t > 0$. Then
$\Theta(u) = 0$ and 
$$
\int_{X \times Y \times S} u d \gamma =  \int_{X \times Y \times S} d\gamma  + t \int_{X \times Y \times S} \varphi d \gamma \to -\infty, \quad t \to +\infty. 
$$
Therefore, $-\Theta^{*}(\gamma) = -\infty$.

If the measure $\gamma$ does not satisfy the condition that $\int_{Y \times S} d\gamma = \int_{X \times S} d\gamma$, 
then there exists a function $h \in C_b(X)$ such that 
$$\int_{X \times Y \times S} h(x) d \gamma < \int_{X \times Y \times S} h(y) d\gamma. $$
Take $$u(x, y, s) = 1 + t \, h(x) - t \, h(y),$$
where $t > 0$.  
Then $\Theta(u) = 0$ and 
\begin{multline*}
\int_{X \times Y \times S} u d \gamma = \int_{X \times Y \times S} d\gamma + \\ + t \Bigl(\int_{X \times Y \times S} h(x) d \gamma - \int_{X \times Y \times S} h(y) d\gamma \Bigr) \to -\infty, \quad t \to +\infty. 
\end{multline*}
Therefore, $-\Theta^*(-\gamma) = -\infty$. 

Suppose that the measure $\gamma$ is nonnegative and satisfies the condition that $\int_{Y \times S} d\gamma = \int_{X \times S} d\gamma$. 
We show that
\begin{equation} \label{eq_theta}
- \Theta^*(-\gamma) = \int_{X \times Y \times S} d \gamma. 
\end{equation}
Letting $u(x, y, s) = 1$, we obtain $- \Theta^*(-\gamma) \le \int_{X \times Y \times S} d \gamma$. 
Let us prove that  $- \Theta^*(-\gamma) \ge \int_{X \times Y \times S} d \gamma$.
Indeed,  if $\Theta(u) < +\infty$, then there exists a function $h \in C_b(X)$ which satifies the inequality (\ref{def_theta}) 
and hence 
\begin{multline*}
\int_{X \times Y \times S} u(x, y, s) d\gamma \ge \int_{X \times Y \times S} d \gamma - 
\int_{X \times Y \times S} h(x) d\gamma + \int_{X \times Y \times S} h(y) d\gamma = \\ = \int_{X \times Y \times S} d\gamma. 
\end{multline*}
Thus, $\inf_{u \in U} (\int_{X \times Y \times S} u d\gamma + \Theta(u)) \ge \int_{X \times Y \times S} d\gamma$. 
Therefore, the equality (\ref{eq_theta}) holds true. 

Now we consider the value $\Xi^{*}(\gamma)$. Let us prove that $\Xi^{*}(\gamma) = 0$, if $Pr_{X \times S}(\gamma) \le \pi^{+}$ and $Pr_{X \times S}(\gamma) \le \pi^{-}$, 
and $\Xi^{*}(\gamma) = +\infty$ otherwise. 

We have
\begin{multline*}
\Xi^{*}(\gamma) = \sup_{u \in U} \Bigl(\int_{X \times Y \times S} u d\gamma - \Xi(u)\Bigr) = \\ = \sup \Bigl\{\int_{X \times Y \times S} (f(x, s) + g(y, s)) d \gamma - 
\int_{X \times S} f(x, s) d \pi^+ - \int_{Y \times S} g(y, s) d \pi^{-}: \\
f \in C_b(X \times S), g \in C_b(Y \times S), f \ge 0, g \ge 0 \Bigr\}.
\end{multline*}

If the inequality $Pr_{X \times S}(\gamma) \le \pi^{+}$ is not satisfied, then there exists a function $f \in C_b(X \times S)$, $f \ge 0$, such that
$$
\int_{X \times Y \times S} f(x, s) d\gamma > \int_{X \times S} f(x, s) d\pi^{+}. 
$$
Then for the function $u(x, y, s) = t \, f(x, s)$, where $t > 0$, we obtain
$$
\int_{X \times Y \times S} u d\gamma - \Xi(u) \ge t \Bigl( \int_{X \times Y \times S} f(x, s) d\gamma - \int_{X \times S} f(x, s) d\pi^{+} \Bigr) \to +\infty, \quad t \to +\infty. 
$$
Thus, $\Xi^{*}(\gamma) = +\infty$. Similarly $\Xi^{*}(\gamma) = +\infty$, if the inequality $Pr_{Y \times S}(\gamma) \le \pi^{-}$ is not satisfied. If $Pr_{X \times S}(\gamma) \le \pi^{+}$ and $Pr_{Y \times S}(\gamma) \le \pi^{-}$, then
$$
\sup_{u \in U} \Bigl(\int_{X \times Y \times S} u d\gamma - \Xi(u)\Bigr) = 0.
$$

Therefore, $-\Theta^*(-\gamma) - \Xi^{*}(\gamma) = -\infty$, if $\gamma \notin \Gamma(\pi^{+}, \pi^{-}, c)$, and
$$-\Theta^*(-\gamma) - \Xi(\gamma) = \int_{X \times Y \times S} d\gamma,$$ if
$\gamma \in \Gamma(\pi^{+}, \pi^{-}, c)$.

Thus, 
$$
\max_{\gamma \in \mathcal M(X \times Y \times S)} (-\Theta^*(-\gamma) - \Xi^*(\gamma)) =
\max_{\gamma \in \Gamma(\pi^{+}, \pi^{-}, c)} \int_{X \times Y \times S} d\gamma. 
$$

In the general case of completely regular topological spaces we consider the Stone-Cech compactifications 
$\beta X = \beta Y$ and $\beta S$ of the spaces $X = Y$ and $S$. For compact spaces the duality formula is already proved. 
Note that every measure $\gamma \in \mathcal M_{+}(\beta X \times \beta Y \times \beta S)$ satisfying the conditions  $Pr_{\beta X \times \beta S}(\gamma) \le \pi^{+}$ and $Pr_{\beta Y \times \beta S} (\gamma) \le \pi^{-}$, 
is concentrated on $X \times Y \times S$. Therefore, the maximum in the primal problem over measures $\gamma \in \mathcal M_{+}(\beta X \times \beta Y \times \beta S)$ is equal to the maximum in the primal problem over measures
$\gamma \in \mathcal M_{+}(X \times Y \times S)$. This implies the duality formula for the initial spaces $X = Y$ and $S$. 

\end{proof}

\section{Reduction to a two-dimensional optimization problem with constraints}

In this section we reduce the problem (\ref{prob_cont})-(\ref{eq}) to a two-dimensional optimization problem. 
Further without limitation of generality we assume that $c(s) = 1$ for all $s \in S$. 
Let $\gamma \in \Gamma(\pi^{+}, \pi^{-}, c)$. 
Denote by
$$
\sigma^{+} = Pr_{X \times S}(\gamma), \quad \sigma^{-} = Pr_{Y \times S}(\gamma) 
$$
the projections of the measure $\gamma$ on $X \times S$ and $Y \times S$ respectively. 
By virtue of (\ref{ineq}) we have
$$
\sigma^{+} \le \pi^{+}, \quad \sigma^{-} \le \pi^{-}.
$$
The balance condition (\ref{eq}) implies that $Pr_{X}(\gamma) = Pr_{Y}(\gamma)$.
Therefore, the measures $\sigma^{+}$ and $\sigma^{-}$ satisfy the properties
$$
Pr_{X}(\sigma^{+}) = Pr_{Y}(\sigma^{-}), \quad Pr_{S}(\sigma^{+}) = Pr_{S}(\sigma^{-}).
$$

Consider the following optimization problem for nonnegative measures $\sigma^{+} \in \mathcal M_{+}(X \times S)$, 
$\sigma^{-} \in \mathcal M_{+}(Y \times S)$: 
we aim to maximize the value 
\begin{equation} \label{prob_2d}
\sigma^{+}(X \times S) \to \max 
\end{equation}
under the condition that
\begin{equation} \label{ineq2}
\sigma^{+} \le \pi^{+}, \quad \sigma^{-} \le \pi^{-}, 
\end{equation}
\begin{equation} \label{eq2}
Pr_{X}(\sigma^{+}) = Pr_{Y}(\sigma^{-}), \quad Pr_{S}(\sigma^{+}) = Pr_{S}(\sigma^{-}).
\end{equation}
 
\begin{theorem} \label{th_equiv}
Let $X = Y$ and $S$ be completely regular topological spaces, $\pi^{+} \in \mathcal M_{+}(X \times S)$
and $\pi^{-}\in \mathcal M_{+}(Y \times S)$ be Radon nonnegative measures. 
Then the maximum in the problem (\ref{prob_cont})-(\ref{eq}) is equal to the maximum in the problem (\ref{prob_2d})-(\ref{eq2}). 
For every optimal solution $(\sigma^+, \sigma^{-})$ to the problem (\ref{prob_2d})-(\ref{eq2}) 
there exists an optimal solution $\gamma$ to the problem (\ref{prob_cont}) such that $Pr_{X \times S}(\gamma) = \sigma^{+}$  and $Pr_{Y \times S}(\gamma) = \sigma^{-}$. 
\end{theorem}

\begin{proof}
Note that for given two measures $\sigma^{+} \in \mathcal M_{+}(X \times S)$, 
$\sigma^{-} \in \mathcal M_{+}(Y \times S)$ satisfying the conditions (\ref{ineq2}) and (\ref{eq2}), 
with the use of gluing lemma we can construct a measure $\gamma \in \mathcal M_{+}(X \times Y \times S)$ such that 
$Pr_{X \times S}(\gamma) = \sigma^{+}$ and $Pr_{Y \times S}(\gamma) = \sigma^{-}$. 
Then
$$
Pr_{X}(\gamma) = Pr_{X} (\sigma^{+}) = Pr_{Y}(\sigma^{-}) = Pr_{Y}(\gamma)
$$ 
and, therefore, $\gamma \in \Gamma(\pi^{+}, \pi^{-}, c)$. 
This implies that the maximum in the problem (\ref{prob_cont})-(\ref{eq}) is equal to the maximum in the problem (\ref{prob_2d})-(\ref{eq2}), and moreover, a measure $\gamma$ is optimal for the problem (\ref{prob_cont})-(\ref{eq}) 
if and only if measures $\sigma^{+}$, $\sigma^{-}$ form an optimal solution to the problem (\ref{prob_2d})-(\ref{eq2}). 
\end{proof}

Thus, the initial (three-dimensional) optimization problem for a measure $\gamma \in \mathcal M_{+}(X \times Y \times S)$ 
can be reduced to (two-dimensional) optimization problem for measures $\sigma^{+}$, $\sigma^{-}$. \\

Note that if measures $\sigma^{+}$ and $\sigma^{-}$ satisfy the conditions (\ref{ineq2}) and (\ref{eq2}), then 
$Pr_{X}(\sigma^{+})  \le Pr_{X}(\pi^{+}) \land Pr_{Y}(\pi^{-})$ and
$Pr_{S}(\sigma^{+})  \le Pr_{S}(\pi^{+}) \land Pr_{S}(\pi^{-})$. 

Denote
$$
\widehat \mu = Pr_{X}(\pi^{+}) \land Pr_{Y}(\pi^{-}), \quad \widehat \nu = Pr_{S}(\pi^{+}) \land Pr_{S}(\pi^{-}). 
$$
Therefore, we obtain

\begin{corollary}
The maximum in the problem (\ref{prob_cont})-(\ref{eq}) does not exceed $\min(\widehat \mu(X), \widehat \nu(S))$.
\end{corollary}

Furthermore, Theorem \ref{th_equiv} implies a necessary and sufficient condition under which it is possible to construct a measure $\gamma$ satisfying the demand and supply for all trading participants. 

\begin{corollary}
The problem (\ref{prob_cont})-(\ref{eq}) has an optimal solution for which the inequalities (\ref{eq}) become equalities  
(that is, the demand and supply of every participant are completely satisfied) if and only if measures 
$\pi^{+}$ and $\pi^{-}$ have the same projections:
$$
Pr_{X}(\pi^{+}) = Pr_{Y}(\pi^{-}), \quad  Pr_{S}(\pi^{+}) = Pr_{S}(\pi^{-}). 
$$
\end{corollary}

Consider the following example where is is possible to find the optimal solution to the problem (\ref{prob_cont})-(\ref{eq}) in an explicit form. 

\begin{example}
{\rm 
Let $X = Y = S = [0, 1]$. Let $\pi^{+}$ be Lebesgue measure on the domain  
$$D = \{(x, s) \in [0, 1]^2: x/2 \le s \le (x+1)/2\},$$  
and let $\pi^{-}$ be Lebesgue measure on the domain $[0, 1]^2 \setminus D$. 
Then
$$Pr_X(\pi^{+}) = Pr_Y(\pi^{-}) = 1/2 \lambda,$$ 
$$Pr_S(\pi^+) = (1 - |2 s - 1|) \lambda, \quad Pr_S(\pi^{-}) = |2s - 1| \lambda, $$
where $\lambda$ is Lebesgue measure on the interval $[0, 1]$. 
We have
$$
\widehat \mu = 1/2 \lambda, \quad \widehat \nu = \min(1 - |2s - 1|, |2s - 1|) \lambda. 
$$
Then $\widehat \nu(S) = 1/4$. 
It is easy to see that there exist measures $\sigma^+ \le \pi^+$ and $\sigma^- \le \pi^-$ with common projections 
$|x - 1/2| \lambda$ and $\widehat \nu$ on the spaces $X$ and $S$ respectively:
$\sigma^+$ is Lebesgue measure on the domain $D \setminus \tilde D$ and $\sigma^-$ is Lebesgue measure on the domain 
$\tilde D \setminus D$, where $\tilde D$ is the image of the domain $D$ under the reflection in the line $x = 1/2$. 
We can construct the optimal measure $\gamma$ on $X \times Y \times S$: the measure $\gamma$ is concentrated on the set
$\{y = 1 - x\}$ and has projections $\sigma^+$ and $\sigma^-$ on $X \times S$ and $Y \times S$ respectively
(we can define $\gamma$ as the image of the measure $\sigma^+$ under the mapping $(x, s) \mapsto (x, 1 - x, s)$). 
That is, the participants $x$ and $1 - x$ (where $x \in [0, 1/2]$) exchange goods with each other: the participant $x$ 
gives the participant $1 - x$ goods $s$, where $x/2 \le s \le (1 - x)/2$, and receives from the participant $1 - x$ goods
$s$, where $1 - x/2 \le s \le (x + 1)/2$.  
}
\end{example}
 
Now we investigate the optimization problem (\ref{prob_2d})-(\ref{eq2}). Denote
$$
\mu = Pr_X(\sigma^{+}) = Pr_Y(\sigma^{-}), \quad \nu = Pr_S(\sigma^{+}) = Pr_S(\sigma^{-}). 
$$ 

In order to solve the problem (\ref{prob_2d})-(\ref{eq2}) we need to find measures $\mu \in \mathcal M_{+}(X)$ and $\nu \in \mathcal M_{+}(S)$ with the maximum possible mass $\mu(X) = \nu(S)$, such that there exist measures 
$\sigma^{+}, \sigma^{-} \in \Pi(\mu, \nu)$ satisfying the conditions $\sigma^{+} \le \pi^{+}$ and $\sigma^{-} \le \pi^{-}$. 

The problem of optimization over the set of measures $\sigma$ with given projections $\mu$ and $\nu$ and an additional constraint $\sigma \le \pi$ has been considered in \cite{KM2}, \cite{KM3}, \cite{P24}. 
In \cite{KM3} the following necessary and sufficient condition for the existence of a measure $\sigma \le \pi$ with given projections $\mu$ and $\nu$ has been formulated. 

\begin{theorem} [\cite{KM3}] \label{th_exist1}
For given measures $\mu \in \mathcal M_{+}(X)$, $\nu \in \mathcal M_{+}(S)$ and $\pi \in \mathcal M_{+}(X \times S)$ 
there exists a measure $\sigma \in \mathcal M_{+}(X \times S)$ such that $Pr_{X}(\sigma) = \mu$, $Pr_{S}(\sigma) = \nu$ and $\sigma \le \pi$ if and only if
$$
\int_X u(x) \mu(dx) + \int_S v(s) \nu(ds) \le \int_{X \times S} [u(x) + v(s)]_{+} \pi(dx \, ds)
$$
for all functions $u \in C_b(X)$, $v \in C_b(S)$, where $[\cdot]_{+} = \max(\cdot, 0)$. 
\end{theorem}

Using Theorem \ref{th_exist1}, we can obtain the following formula for the optimal value in the problem (\ref{prob_cont})-({\ref{eq}). 

\begin{theorem} \label{th_opt}
Let $X$ and $S$ be completely regular topological spaces, let $\pi^{+} \in \mathcal M_{+}(X \times S)$ and $\pi^{-}\in \mathcal M_{+}(X \times S)$ be Radon nonnegative measures. 
Then the maximum in the problem (\ref{prob_cont})-(\ref{eq}) is equal to 
\begin{equation} \label{opt}
 \inf_{u \in C_b(X), v \in C_b(S)} \Bigl\{\int_{X \times S} [u(x) + v(s)]_{+} \pi^{+}(dx \, ds) + \int_{X \times S} [1 - u(x) - v(s)]_{+} \pi^{-}(dx \, ds) \Bigr\}. 
\end{equation}
\end{theorem}

\begin{proof}
Let us define functionals $\Phi^{+}$ and $\Phi^{-}$ on $C_b(X) \times C_b(S)$:  
$$
\Phi^{+}(u, v) = \int_{X \times S} [u(x) + v(s)]_{+} \pi^{+}(dx \, ds), 
$$
$$
\Phi^{-}(u, v) = \int_{X \times S} [u(x) + v(s)]_{+} \pi^{-}(dx \, ds)
$$
for all $u \in C_b(X)$, $v \in C_b(S)$. 
Note that $\Phi^{+}$ and $\Phi^{-}$ are continuous sublinear functionals on $C_b(X) \times C_b(S)$. 

We first consider the case where $X$ and $S$ are compact topological spaces. Then by Riesz theorem every continuous linear functional $L$ on $C_b(X) \times C_b(S)$ can be represented in the form  
$$
L(u, v) = \int_X u \, d\mu + \int_S v \, d \nu
$$
for some (signed) measures $\mu \in \mathcal M(X)$, $\nu \in \mathcal M(S)$. 
Moreover, the functional $L$ is nonnegative (that is, $L(u, v) \ge 0$ for all pairs of nonnegative functions  
$u \in C_b(X)$ and $v \in C_b(S)$) if and only if measures $\mu$ and $\nu$ are nonnegative, i.e.  
$\mu \in \mathcal M_{+}(X)$, $\nu \in \mathcal M_{+}(S)$. 

By virtue of Theorem \ref{th_exist1} the problem of finding measures $\mu \in \mathcal M_{+}(X)$ and $\nu \in \mathcal M_{+}(S)$ with maximum possible mass $\mu(X) = \nu(S)$ for which 
there exist measures $\sigma^{+}, \sigma^{-} \in \Pi(\mu, \nu)$ such that $\sigma^{+} \le \pi^{+}$ and $\sigma^{-} \le \pi^{-}$, is equivalent to finding a nonnegative continuous linear functional $L$ on 
$C_b(X) \times C_b(S)$  satisfying the inequalities $L \le \Phi^{+}$ and $L \le \Phi^{-}$, for which the value $L(1, 1)$ is maximal.  

Note that if $L$ is a linear functional on $C_b(X) \times C_b(S)$ satisfying the conditions $L \le \Phi^{+}$ and $L \le \Phi^{-}$, then $L$ is nonnegative and continuous. 
Indeed, for any nonnegative functions $u \in C_b(X)$ and $v \in C_b(S)$ we have 
$L(u, v) = -L(-u, -v) \ge -\Phi^{+}(-u, -v) = 0$. 
Furthermore, for any functions $u \in C_b(X)$ and $v \in C_b(S)$ it holds that $|L(u, v)| \le \Phi^{+}(u, v) + \Phi^{+}(-u, -v)$, 
which implies the continuity of the linear functional $L$ due to the continuity of $\Phi^{+}$. 

Let
\begin{multline} \label{def_phi}
\Phi(u, v) = \inf \bigl\{\Phi^{+}(u^{+}, v^{+}) + \Phi^{-}(u^{-}, v^{-}): u = u^{+} + u^{-}, v = v^{+} + v^{-}, \\ 
u^{+}, u^{-} \in C_b(X), v^{+}, v^{-} \in C_b(S) \bigr\} 
\end{multline}
for all $u \in C_b(X)$, $v \in C_b(S)$.

Let us show that if a continuous linear functional $L$ satisfies the conditions $L \le \Phi^{+}$ and $L \le \Phi^{-}$, 
then $L \le \Phi$.  
Let $u \in C_b(X)$ and $v \in C_b(S)$. Then for any functions $u^{+}, u^{-} \in C_b(X)$, $v^{+}, v^{-} \in C_b(S)$ 
such that $u = u^{+} + u^{-}$ and $v = v^{+} + v^{-}$, 
we have
$$
L(u, v) = L(u^{+}, v^{+}) + L(u^{-}, v^{-}) \le \Phi^{+}(u^{+}, v^{+}) + \Phi^{-}(u^{-},  v^{-}). 
$$ 
Thus, $L(u, v) \le \Phi(u, v)$. 

We prove that the functional $\Phi$ is sublinear, that is,  
$$\Phi(u_1 + u_2, v_1 + v_2) \le \Phi(u_1, v_1) + \Phi(u_2, v_2)$$ 
for all $u_1, u_2 \in C_b(X)$, $v_1, v_2 \in C_b(S)$. 
Indeed, for any  $u_i^{+}, u_i^{-} \in C_b(X)$ and $v_i^{+}, v_i^{-} \in C_b(S)$ such that $u_i = u_i^{+} + u_i^{-}$, $v_i = v_i^{+} + v_i^{-}$ for all $i \in \{1, 2\}$, we have
\begin{multline*}
\Phi(u_1 + u_2, v_1 + v_2)  \le \Phi^{+}(u_1^{+} + u_2^{+}, v_1^{+} + v_2^{+}) + \Phi^{-}(u_1^{-} + u_2^{-}, v_1^{-} + v_2^{-}) \le \\ \le
\Phi^{+}(u_1^{+}, v_1^{+}) + \Phi^{+}(u_2^{+}, v_2^{+}) + \Phi^{-}(u_1^{-}, v_1^{-}) + \Phi^{-}(u_2^{-}, v_2^{-}), 
\end{multline*}
since the functionals $\Phi^{+}$ and $\Phi^{-}$ are sublinear. 
Taking the infimum over all such functions $u_i^{+}, u_i^{-}$ and $v_i^{+}, v_i^{-}$, $i \in \{1, 2\}$, 
we obtain that $\Phi(u_1 + u_2, v_1 + v_2) \le \Phi(u_1, v_1) + \Phi(u_2, v_2)$. 

By Hahn-Banach theorem for the sublinear functional $\Phi$ there exists a linear functional $L$ on $C_b(X) \times C_b(S)$ such that $L \le \Phi$ and $L(1, 1) = \Phi(1, 1)$. 
Then, as mentioned above, from the inequality $L \le \Phi$ it follows that the linear functional $L$ is continuous and nonnegative. 
Therefore, $L$ can be represented in the form $L(u, v) = \int u d\mu + \int v d \nu$ for some nonnegative measures $\mu \in \mathcal M_{+}(X)$, $\nu \in \mathcal M_{+}(S)$. 
Then by Theorem \ref{th_exist1} for measures $\mu$ and $\nu$ it is possible to find measures $\sigma^+ \le \pi^+$ and $\sigma^- \le \pi^-$ with projections $\mu$ and $\nu$ on $X$ and $S$ respectively. 
Using gluing lemma, we can construct a measure $\gamma$ on $X \times Y \times S$ which gives the optimal solution to the problem (\ref{prob_cont})-(\ref{eq}). 
Therefore, the maximum in the problem (\ref{prob_cont})-(\ref{eq}) is equal to $\mu(X) = \nu(S) = L(1, 1)/2 = \Phi(1, 1) /2$
and by definition (\ref{def_phi}) this implies the equality (\ref{opt}). 

The general case of completely regular topological spaces can be reduced to the case of compact topological spaces 
using the Stone-Cech compactifications $\beta X$ and $\beta S$ of the spaces $X$ and $S$. 

\end{proof}

In \cite{KM3} another equivalent condition for the existence of a measure $\sigma \le \pi$ with given projections $\mu$ and $\nu$ is also given. 
Denote
$$
\Pi(\mu, \nu; \pi) = \{\sigma \in \Pi(\mu, \nu): \sigma \le \pi\}. 
$$

\begin{theorem}[\cite{KM3}] \label{th_exist2}
Let $\mu \in \mathcal M_{+}(X)$, $\nu \in \mathcal M_{+}(S)$ and $\pi \in \mathcal M_{+}(X \times S)$. The set $\Pi(\mu, \nu; \pi)$ is non-empty if and only if
$$
\mu(A) + \nu(B) \le \alpha + \pi(A \times B) 
$$
for all Borel sets $A \subset X$, $B \subset S$, where $\alpha = \mu(X) = \nu(S)$.
\end{theorem}

Let us denote
$$
r(A, B) = \min(\pi^{+}(A \times B), \pi^{-}(A \times B)) 
$$
for all Borel sets $A \subset X$, $B \subset S$. 
Then the problem (\ref{prob_2d})-(\ref{eq2}) can be reduced to finding such measures $\mu \in \mathcal M_{+}(X)$ and $\nu \in \mathcal M_{+}(S)$ with maximum possible total mass $\mu(X) = \nu(S) = \alpha$ such that 
\begin{equation} \label{cond_r}
\mu(A) + \nu(B) \le \alpha + r(A, B) 
\end{equation}
for all Borel sets $A \subset X$, $B \subset S$. 

The condition (\ref{cond_r}) may be rewritten in the following equivalent form:
\begin{equation} \label{cond_r2}
\mu(A) \le \nu(S \setminus B) + r(A, B), \quad \nu(B) \le \mu(X \setminus A) + r(A, B). 
\end{equation}

Hence the optimal solution to the problem of optimal exchange can be constructed in the following way.  

Step 1. We find measures $\mu \in \mathcal M_{+}(X)$ and $\nu \in \mathcal M_{+}(S)$ with maximum possible mass 
for which the condition (\ref{cond_r2}) is satisfied. 

Step 2. Using Theorem \ref{th_exist2}, we find measures $\sigma^{+}, \sigma^{-} \in \Pi(\mu, \nu)$ such that 
$\sigma^{+} \le \pi^{+}$ and $\sigma^{-} \le \pi^{-}$.

Step 3. Applying gluing lemma for measures $\sigma^{+}$ and $\sigma^{-}$, we construct a measure $\gamma$ 
which provides an optimal solution to the problem of optimal exchange.  

\section{Reduction to a Kantorovich problem with density constraints}

In this section we will assume that the measures $\pi^{+}$ and $\pi^{-}$ are concentrated on disjoint sets 
$A^{+}$ and $A^{-}$, where $A^{+}, A^{-} \subset X \times S$ are Borel sets. 
This assumption means that for every trading participant the set of goods which he is ready to share with other participants
and the set of goods which he wishes to receive do not intersect. 

We show that in this case the problem of optimal exchange can be reduced to a Kantorovich problem with density constraints. 

Let  
$$
\mu^{+} = Pr_X \pi^{+}, \, \nu^{+} = Pr_S \pi^{+}, \quad \pi = \pi^{+} + \pi^{-}. 
$$ 

We define the cost function $h \colon X \times S \to \mathbb R$ in the following way: 
$h(x, s) = 1$ for all $(x, s) \in A^{+}$ and $h(x, s) = 0$ otherwise. 

Consider the following Kantorovich optimal transportation problem with density constraints: 
we aim to minimize the integral
\begin{equation} \label{opt_transport}
\int_{X \times S} h(x, s) \tau(dx \, ds) \to \inf, \quad \tau \in \Pi(\mu^{+}, \nu^{+}; \pi), 
\end{equation}
over all measures $\tau$ from the set
$\Pi(\mu^{+}, \nu^{+}; \pi) = \{\tau \in \Pi(\mu^{+}, \nu^{+}): \tau \le \pi\}$. 

Denote by $K_h(\mu^{+}, \nu^{+}; \pi)$ the infimum in the Kantorovich problem (\ref{opt_transport}). 

\begin{theorem}
The maximum in the problem (\ref{prob_2d})-(\ref{eq2}) is equal to 
$$\pi^{+}(X \times S) - K_h(\mu^{+}, \nu^{+}; \pi). $$ 
The measures $(\sigma^{+}, \sigma^{-})$ constitute an optimal solution to the problem (\ref{prob_2d})-(\ref{eq2}) 
if and only if the measure $\tau = \pi^{+} - \sigma^{+} + \sigma^{-}$ 
is an optimal solution to the problem (\ref{opt_transport}). 
\end{theorem}

\begin{proof}
We build a one-to-one correspondence between pairs of measures 
$\sigma^{+}, \sigma^{-} \in \mathcal M_{+}(X \times S)$ satisfying the conditions (\ref{ineq2}) and (\ref{eq2})
and measures $\tau \in \Pi(\mu^{+}, \nu^{+}; \pi)$, such that
\begin{equation} \label{equal}
\int_{X \times S} h(x, s) \tau(dx \, ds) = \pi^{+}(X \times S) - \sigma^{+}(X \times S), 
\end{equation}
if a measure $\tau$ corresponds to the pair of measures $(\sigma^{+}, \sigma^{-})$. 

Let measures $\sigma^{+}, \sigma^{-} \in \mathcal M_{+}(X \times S)$ satisfy the conditions (\ref{ineq2}) and (\ref{eq2}). 
Set $$\tau = \pi^{+} - \sigma^{+} + \sigma^{-}. $$
Then $\tau \ge 0$, because $\sigma^{+} \le \pi^{+}$. 
Since measures $\sigma^{+}$ and $\sigma^{-}$ have equal projections on $X$ and $S$, 
we have $Pr_X \tau = Pr_X \pi^{+} = \mu^{+}$ and $Pr_S \tau = Pr_S \pi^{+} = \nu^{+}$, that is, $\tau \in \Pi(\mu^{+}, \nu^{+})$. 
Moreover, $\tau \le \pi = \pi^{+} + \pi^{-}$, since $\sigma^{-} \le \pi^{-}$. 
Therefore, $\tau \in \Pi(\mu^{+}, \nu^{+}; \pi)$. 

Furthermore, we have
$$
\int_{X \times S} h(x, s) \tau(dx \, ds) = \tau(A^{+})  = \pi^{+}(A^{+}) - \sigma^{+}(A^{+})
$$
since $\sigma^{-} \le \pi^{-}$ and $\pi^{-}(A^{+}) = 0$. The last expression equals
$\pi^{+}(X \times S) - \sigma^{+}(X \times S)$, 
because $\sigma^{+} \le \pi^{+}$ and the measure $\pi^{+}$ is concentrated on $A^{+}$. 

Conversely, let $\tau \in \Pi(\mu^{+}, \nu^{+}; \pi)$. We set 
$$\sigma^{+} = \pi^{+} - \tau|_{A^{+}}, \quad \sigma^{-} = \tau|_{A^{-}}. $$
Since $\tau \le \pi$, we have $\tau|_{A^{+}} \le \pi|_{A^+} = \pi^{+}$ and $\tau|_{A^{-}} \le \pi|_{A^{-}} = \pi^{-}$. 
Therefore, $\sigma^{+}$ and $\sigma^{-}$ are nonnegative measures and $\sigma^{+} \le \pi^{+}$, 
$\sigma^{-} \le \pi^{-}$. 

Let us show that $Pr_X \sigma^{+} = Pr_X \sigma^{-}$ and $Pr_S \sigma^{+} = Pr_S \sigma^{-}$. 
Indeed, the difference of measures $\sigma^{+} - \sigma^{-} = \pi^{+} - \tau$ has zero projections on $X$ and $S$, 
since $\tau \in \Pi(\mu^{+}, \nu^{+})$. 

Thus, the measures $\sigma^{+}$ and $\sigma^{-}$ satisfy the conditions (\ref{ineq2}) and (\ref{eq2}). 

The constructed correspondences between pairs of measures $(\sigma^{+}, \sigma^{-})$ and measures $\tau$ 
are mutually inverse and satisfy the equality (\ref{equal}). 
Therefore, the maximum in the problem (\ref{prob_2d})-(\ref{eq2}) is equal to 
$$\pi^{+}(X \times S) - K_h(\mu^{+}, \nu^{+}; \pi),$$ 
and the optimality of measures $(\sigma^{+}, \sigma^{-})$ for the problem (\ref{prob_2d})-(\ref{eq2}) 
is equivalent to the optimality of the measure $\tau = \pi^{+} - \sigma^{+} + \sigma^{-}$ for the problem (\ref{opt_transport}). 

\end{proof}

\end{document}